\documentclass[11pt,reqno]{amsart} 

\usepackage{amssymb,latexsym}
\usepackage{cite} 

\usepackage[height=190mm,width=130mm]{geometry} 

\theoremstyle{plain}
\newtheorem{theorem}{Theorem}

\theoremstyle{definition}
\newtheorem{definition}{Definition}

\theoremstyle{remark}
\newtheorem{remark}{Remark}

\numberwithin{equation}{section} 

\begin{document}
\title[Geodesic mappings of (pseudo-) Riemannian manifolds]{Geodesic mappings of (pseudo-)
 Riemannian manifolds preserve the class of differentiability} 

\author{I. Hinterleitner}
\address{Brno University of Technology \\ Faculty of Civil Engineering \\ Dept.~of Mathe\-ma\-tics\\ \v Zi\v zkova 17\\  602 00 \\ Brno \\ Czech Republic}

\email{hinterleitner.irena@seznam.cz}

\thanks{The paper was supported by grant P201/11/0356 of The Czech Science Foundation and by the project FAST-S-12-25 of the Brno University of Technology.}

\author{J. Mike\v s}

\address{Palacky University\\Faculty of Science\\ Dept.~of Algebra and Geometry\\
17.~listopadu 12\\ 77146 \\ Olomouc\\ Czech Republic}

\curraddr{KAG PrF UP\\
17. listopadu 12\\ 77146 \\ Olomouc\\ Czech Republic}

\email{josef.mikes@upol.cz}

\begin{abstract}
In this
 paper we prove that geodesic mappings of (pseudo-) Riemannian manifolds preserve the class of differentiability \hbox{$(C^r, r\geq1)$}.
Also, if  the Einstein space $V_n$ admits a non trivial geodesic mapping onto  a \hbox{(pseudo-)} Riemannian manifold $\bar V_n\in C^1$, then $\bar V_n$ is an Einstein space.
 If  a four-dimensional Einstein space with  non constant curvature globally admits a geodesic mapping onto a (pseudo-) Riemannian manifold $\bar V_4\in C^1$, then the mapping is  affine and, moreover, if the scalar curvature is non vanishing, then the mapping is homothetic, i.e. $\bar g={\rm const}\cdot g$.
\end{abstract}




\maketitle

\def\a{\alpha}
\def\b{\beta}
\def\g{\gamma}
\def\G{\Gamma}
\def\d{\delta}
\def\la{\lambda}
\def\e{{\rm \,e\,}}
\def\vn{$V_n$}
\def\vnn{$\bar V_n$}
\def\nad#1#2{\buildrel{#1} \over{#2}\!\!\strut}
\def\pod#1#2{\mathrel{\mathop{#2}\limits_{#1}}\strut}
\def\ds{\displaystyle}
\def\noi{\noindent}

\section{Introduction}

The paper is devoted to the geodesic mapping theory of (pseudo-) Riemannian manifolds with respect to differentiability of their metrics.
Most of the results in this area are formulated for ``sufficiently" smooth, or analytic, geometric objects, as usual in differential geometry.
It can be observed in most of monographs and researches, dedicated to the study of the theory of geodesic mappings and transformations, see \cite{ami1,cizl12,ei1,ei2,formik,hal11,himi10,himi11,himi12,le,mid,miei,mi,miho,mibe89,mihi10,mkv,mistr,mvh,no,pe,pr,rmg,shir,si,sol,stmi10,th,vra,we}.

Let $V_n=(M,g)$ and $\bar V_n=(\bar M,\bar g)$ be  (pseudo-) Riemannian manifolds, where $M$ and $\bar M$ are \hbox{$n$-di}\-men\-sio\-nal manifolds with dimension $n\geq2$, $g$ and $\bar g$ are metrics. All the manifolds are assumed to be connected.
\begin{definition}
A diffeomorphism $f$: $V_n\to \bar V_n$ is called a \textit{geodesic mapping} of~$V_n$ onto $\bar V_n$ if $f$ maps any geodesic  in $V_n$ onto a geodesic in $\bar V_n$.
\end{definition}
Hinterleitner and Mike\v s \cite{himi12} have proved the following theorem:

\begin{theorem}\label{th1}
If  the (pseudo-) Riemannian manifold $V_n$ $(V_n\in C^r$, $r\geq2$, $n\geq2)$ admits a geodesic mapping onto $\bar V_n\in C^2$, then $\bar V_n$ belongs to $C^r$.
\end{theorem}

Here and later $V_n=(M,g)\in C^r$ denotes that  $g\in C^r$, i.e.~in a coordinate neighborhood $(U,x)$  for the components of the metric $g$ holds $g_{ij}(x)\in C^r$.\linebreak
If~$V_n\in C^r$ then $M\in C^{r+1}$. This means that the atlas on the manifold $M$ has the differentiability  class $C^{r+1}$, i.e. for non disjoint charts $(U,x)$ and $(U',x')$ on $U\cap U'$ it is true that the transformation $x'=x'(x)\in C^{r+1}$.

We suppose that the differentiability class $r$ is equal to $0,1,2,\dots,\infty,\omega$, where $0,\infty$ and $\omega$  denotes continuous, infinitely differentiable, and real analytic functions respectively.

In the paper we proof more general results.
The following theorem holds:
\begin{theorem}\label{th2}
If the (pseudo-) Riemannian manifold $V_n$ $(V_n\in C^r$, $r\geq1$, $n\geq2)$ admits a geodesic mapping onto $\bar V_n\in C^1$, then $\bar V_n$ belongs to $C^r$ .
\end{theorem}

Briefly, this means that \textit{the geodesic mapping preserves the class of smoothness of the metric}.
\begin{remark}
It's easy to proof that the Theorems \ref{th1} and \ref{th2} are valid also for $r=\infty$ and for $r=\omega$. This follows from the theory of solvability of differential equations.
Of course we can apply this theorem only  locally, because differentiability is a local property.
\end{remark}

\begin{remark}
To require $V_n,\,\bar V_n\in C^1$ is a minimal requirement for geodesic mappings.
\end{remark}

T.~Levi-Civita \cite{le} found metrics  ({\it Levi-Civita metrics}) which admit geodesic mappings, see \cite{ami1,ei1}, \cite[p.~173]{mvh}, \cite[p.~325]{pe}. From these metrics we can easily see examples of non trivial geodesic mappings $V_n\to \bar V_n$, where
\begin{itemize}
\item $V_n, \bar V_n\in C^r$ and $\not\in C^{r+1}$ for $r\in \mathbb N$;
\item $V_n, \bar V_n\in C^\infty$ and $\not\in C^\omega$;
\item $V_n, \bar V_n\in C^\omega$.
\end{itemize}
\medskip

\section{Geodesic mappings of Einstein manifolds}

 These results may be applied for geodesic mappings of Einstein manifolds $V_n$ onto pseudo-Riemannian manifolds  $\bar V_n\in C^1$.

Geodesic mappings of Einstein spaces have been studied by many authors
starting by A.Z.~Petrov (see \cite{pe}).  Einstein spaces $V_n$ are
characterized by the condition $Ric={\rm const}\cdot g.$

 An Einstein space $V_3$ is a space of constant curvature.
It is known that Riemannian spaces of constant curvature form a closed class with respect to geodesic mappings (Beltrami theorem \cite{ei1,mkv,mvh,pe,rmg,si}).
 In 1978 (see \cite{miei} and PhD. thesis \cite{mid}, and see \cite{mi,miki03,mihiki06}, \cite[p.~125]{mkv}, \cite[p.~188]{mvh}) Mike\v s
proved that under the conditions $V_n,\bar V_n\in C^3$ the following theorem holds (locally):
\begin{theorem}\label{th3}
If  the Einstein space $V_n$  admits a non trivial geodesic mapping onto a (pseudo-) Riemannian manifold $\bar V_n$, then $\bar V_n$ is an Einstein space.
\end{theorem}

Many properties of Einstein spaces appear when $V_n\in C^3$ and $n>3$.
Moreover, it is known (D.M.\,DeTurck and J.L.\,Kazdan \cite{detu}, see \cite[p.~145]{bess}), that Einstein space \vn\   belongs to $C^\omega$, i.e. for all points of \vn\ a local coordinate system~$x$  exists for which $g_{ij}(x)\in C^\omega$ ({\it analytic coordinate system}).

It implies global validity of  Theorem \ref{th3}, and  on basis of Theorem~\ref{th2} the following more general theorem holds:
\begin{theorem}\label{th4}
If  the Einstein space $V_n$ admits a nontrivial geodesic mapping onto a (pseudo-) Riemannian manifold $\bar V_n\in C^1$, then $\bar V_n$ is an Einstein space.
\end{theorem}

The present Theorem is true globally, because the function $\Psi$ which determines the geodesic mapping is real analytic on an analytic coordinate system  and so $\psi\ (=\nabla\Psi)$  is vanishing only on a set points of zero measure. This simplifies the proof given in \cite{himi12}.

Finally, basing on the results (see \cite{mi,miki82,miki03,mihiki06}, \cite[p.~128]{mkv}, \cite[p.~194]{mvh}) for geodesic mappings of a four-dimensional Einstein manifolds  the following theorem holds:
\begin{theorem}\label{th5}
If  a four-dimensional Einstein space $V_4$ with non constant curvature globally admits a geodesic mapping onto a (pseudo-) Riemannian manifold $\bar V_4\in C^1$, then the mapping is affine and, moreover, if the scalar curvature is non vanishing, then the mapping is homothetic, i.e. $\bar g={\rm const}\cdot g$.
\end{theorem}

\section{Geodesic mapping theory for $V_n\to\bar V_n$ of class $C^1$}

We briefly remind some main facts of geodesic mapping theory of (pseudo-) Riemannian manifolds which were found by T.~Levi-Civita \cite{le}, L.P.~Eisenhart \cite{ei1,ei2} and N.S.~Sinyukov \cite{si}, see \cite{ami1,himi10,himi11,himi12,mid,mi,mibe89,mihi10,mkv,mvh,no,pe,pr,
rmg,shir,si,sol,th,vra,we}.
In these results  no details about the smoothness class of the metric were stressed. They were formulated ``for sufficiently smooth'' geometric objects.

Since a geodesic mapping $f$: $V_n\to\bar V_n$ is a diffeomorphism, we can suppose  $\bar M=M$.
A (pseudo-) Riemannian manifold $V_n=(M,g)$ admits a geodesic mapping onto $\bar V_n=(M,\bar g)$ if and only if the \textit{Levi-Civita equations} \begin{equation}\label{eq1}
\bar\nabla_XY=\nabla_XY+\psi(X)Y+\psi(Y)X
\end{equation}
hold for any tangent fields $X,Y$ and where $\psi$ is a differential form on $M$.
Here $\nabla$ and $\bar \nabla$ are Levi-Civita connections of $g$ and $\bar g$, respectively. If $\psi\equiv 0$ then $f$ is {\it affine} or {\it  trivially geodesic}.

Let $(U,x)$ be a chart from the atlas on $M$. Then equation \eqref{eq1} on $U$ has the following
 local form:
$
\bar\Gamma^h_{ij}=\Gamma_{ij}^h+\psi_i\delta^h_j+\psi_j\delta_i^h,
$
\ where $\Gamma^h_{ij}$ and $\bar\Gamma^h_{ij}$ are the Christoffel symbols of $V_n$ and $\bar V_n$, $\psi_i$~are components of $\psi$ and $\delta^h_i$ is the Kronecker delta.
Equations \eqref{eq1} are equivalent to the following Levi-Civita equations
  \begin{equation}\label{eq2}
 \nabla_k\bar g_{ij}=2\psi_k \bar g_{ij}+\psi_i\bar g_{jk}+\psi\bar g_{ik}
  \end{equation}
where    $\bar g_{ij}$ are components of $\bar g$.

It is known that  \
$$\ds\psi_i=\partial_i\Psi,\quad
\ds\Psi=\frac{1}{2(n+1)}\ln\left|\frac{\det\bar g}{\det g}\right|,
\quad \partial_i=\ds\frac{\partial\ }{\partial x^i}.
$$

N.S.~Sinyukov  proved that the Levi-Civita equations \eqref{eq1} and \eqref{eq2}  are
equivalent to (\cite[p.~121]{si}, \cite{mi}, \cite[p.~108]{mkv}, \cite[p.~167]{mvh}, \cite[p.~63]{rmg}):
  \begin{equation}\label{eq3}
  \nabla_k a_{ij}=\lambda_i g_{jk}+\lambda_j g_{ik},
  \end{equation}
  where \\[-6mm]
\begin{equation}\label{eq4}
\hbox{(a) \ } a_{ij}=\e^{2\Psi}\bar g^{\alpha\beta}g_{\alpha i}g_{\beta j};\quad
\hbox{(b) \ } \lambda_i=-\e^{2\Psi}\bar g^{\alpha\beta}g_{\beta i}\psi_{\alpha}.
\end{equation}
 From \eqref{eq3} follows
 $\lambda_i=\partial_i(\frac12\, a_{\a\b}g^{\a\b})$,
 $(g^{ij})=(g_{ij})^{-1}$ and  $(\bar g^{ij})=(\bar g_{ij})^{-1}$.

 On  the other hand \cite[p.~63]{rmg}:
\begin{equation}\label{eq5}
\bar g_{ij}=\e^{2\Psi}\hat g_{ij},\quad
  \ds\Psi=\frac12\ln\left|\frac{\det\hat g}{\det g}\right|,
  \quad
  (\hat g_{ij})=(g^{i\a}g^{j\b}a_{\a\b})^{-1}.
\end{equation}

Equations \eqref{eq3} and  \eqref{eq4} we can rewrite in the following equivalent form (see \cite{mibe89}, \cite[p.~150]{mvh}):
\begin{equation}\label{eq6}
\nabla_k a^{ij}=\lambda^i\delta^j_k+\lambda^j\delta^i_k,
\end{equation}
where
\begin{equation}\label{eq7}
\hbox{(a) \ } a^{ij}=\e^{2\Psi}\bar g^{ij}\quad \mbox{and}\quad
\hbox{(b) \ } \lambda^i=-\psi_{\alpha}a^{\alpha i}.
\end{equation}
Evidently, it  follows
\begin{equation}\label{eq8}
\lambda^i=\frac 12\ g^{ik}\,\partial_k(a^{\a\b}g_{\a\b}).
\end{equation}

The above formulas \eqref{eq1}, \eqref{eq2}, \eqref{eq3},  \eqref{eq6}, are the criterion for geodesic mappings $V_n\to
\bar V_n$ globally as well as locally.
These formulas are true only under the condition
 \vn, \vnn\ $\in C^1$.

\section{Geodesic mapping theory for $V_n\in C^2\to\bar V_n\in C^1$ }

In this section, we prove the above main Theorem \ref{th2}.
It is easy to see that  Theorem \ref{th2} follows from  Theorem \ref{th1} and the following theorem.
\begin{theorem}\label{thm:3}
If $V_n\!\in\! C^2$ admits a geodesic mapping onto $\bar V_n\in C^1$, then \hbox{$\bar V_n\in C^2$}.
\end{theorem}
\begin{proof}
Below we prove   Theorem \ref{thm:3}.\smallskip

{\bf 3.1}
We will suppose that the (pseudo-) Riemannian manifold $V_n\in C^2$ admits the geodesic mapping onto the (pseudo-) Riemannian manifold $\bar V_n\in C^1$.
Furthermore, we can assume that $\bar M=M$.

We study the coordinate neighborhood $(U,x)$ of any point $p=(0,0,\dots,0)$ at $M$.
Evidently, components $g_{ij}(x)\in C^2$ and $\bar g_{ij}(x)\in C^1$ on $U\subset M$. On~$(U,x)$ formulas \eqref{eq1}--\eqref{eq8} hold.
From that facts it follows that the functions $g^{ij}(x)\in C^2$,
$\bar g^{ij}(x)\in C^1$, $\Psi(x)\in C^1$, $\psi_i(x)\in C^0$, $a^{ij}(x)\in C^1$,  $\lambda^i(x)\in C^0$, and $\G^h_{ij}(x)\in C^1$, where $\G^h_{ij}=\frac12\ g^{hk}(\partial_ig_{jk}+\partial_jg_{ik}-\partial_kg_{ij})$ are  Christoffel symbols.
\smallskip

{\bf 3.2}
It is easy to see that in a neighborhood of the point $p$ in $V_n\ \in C^r$  exists a semigeodesic coordinate system $(U,x)$ for which the metric $g \in C^r$ has the following form (see \cite{ei1}, \cite[p.~64]{mvh})
\begin{equation}\label{eq9}
{\rm d}s^2=e({\rm d}x^1)^2+g_{ab}(x^1,\ldots, x^n){\rm d}x^a{\rm d}x^b,\quad
e=\pm 1,\quad a,b>1.\end{equation}
Evidently, for $a>1$:
\begin{equation}\label{eq9a}
g_{11}=g^{11}=e=\pm 1,\quad g_{1a}=g^{1a}=0 \hbox{ \ and \ } \G^1_{11}=\G^1_{1a}=\G^a_{11}=0 .
\end{equation}
We can construct such a coordinate system  using a coordinate transformation of class $C^{r+1}$ for a basis of  non-isotropic hypersurfaces $\Sigma\in C^{r+1}$ in a neighborhood of $p\in \Sigma$.
Moreover, we can assume at $p$ that
\begin{equation}\label{10}
g_{ij}(0)=e_i\,\d_{ij}; \quad e_i=\pm1.
\end{equation}

{\bf 3.3}
Equations \eqref{eq6} we write in the following form
\begin{equation}\label{eq12}
\partial_ka^{ij}=\lambda^i\d^j_k+\lambda^j\delta^i_k
-a^{i\a}\G^j_{\a k}-a^{j\a}\G^i_{\a k}.
\end{equation}
Because $a^{ij}\in C^1$ and $\G^j_{\a k}\in C^1$ from equation \eqref{eq12} we have the existence of the derivative immediately
$$
\partial_{kl}a^{ii},\
\partial_{kk}a^{ii}, \
\partial_{ki}a^{ii}(\equiv\partial_{ik}a^{ii}),\
\partial_{kl}a^{ij}, \
\partial_{kk}a^{ij}, \
\partial_{ki}a^{ij}(\equiv\partial_{ik}a^{ij}),
$$
for each set of different indices $i,j,k,l$.
Derivatives do not depend on the order, because they are continuous functions.

We compute formula \eqref{eq12} for $i=j=k$ and for $i\neq j=k$:
$$
\partial_ia^{ii}=2\la^i-2a^{i\a}\G^i_{\a i}
\hbox{\ \ and \ \ }
\partial_ka^{ik}=\la^i-a^{k\a}\G^i_{\a k} -a^{i\a}\G^k_{\a k}
$$
where for an index $k$ we do not carry out the Einstein summation, and after eliminating $\la^i$ we obtain
\begin{equation}\label{eq14}
\hbox{$\frac12$}\ \partial_ia^{ii}-\,\partial_ka^{ik}=a^{k\a}\G^i_{\a k} +a^{i\a}\G^k_{\a k}-a^{i\a}\G^i_{\a i}
\end{equation}
Because there exists the partial derivative $\partial_{ik}a^{ii}$,
formula \eqref{eq14} implies the existence of the partial derivatives $\partial_{kk}a^{ik}.$\medskip

{\bf 3.4}
In the semigeodesic  coordinate system   \eqref{eq9} we compute \eqref{eq12} for $i=j=k=1$:
$\lambda^1=\frac12\ \partial_1a^{11}$, and from  \eqref{eq8}:
$\lambda^1=\hbox{$\frac12$}\ \partial_1(a^{11}+ea^{\a\b}g_{\a\b})$, we obtain $\partial_1(a^{\a\b}g_{\a\b})=0$. Here and later $\a,\b>1$.

Further \eqref{eq12} for $i=j=1$ and $k=2$ we have the following expression
\
$
\partial_1a^{12}+a^{1\g}\G^2_{\g 1}+a^{2\g}\G^1_{\g 1}=\lambda^2
$.
\
Using \eqref{eq8} we have
$$
\partial_1 a^{12}=\hbox{$\frac12$}\ g^{2\g}\cdot\partial_\g(a^{11}+a^{\a\b} g_{\a\b})-a^{1\g}\G^2_{\g1}, \quad \g>1,
$$
and after  integration  we obtain \\[1pt]

\centerline{$
a^{12}=\frac 12\left(\int_0^{x^1}g^{2\g}(\tau^1,x^2,\ldots,x^n)d\tau^1\right)\cdot \partial_\g(a^{\a\b}\cdot g_{\a\b})+
$}

\ \\[-15mm]

\begin{equation}\label{eq16}\ \end{equation}
\\[-10mm]

\centerline{$
\frac 12\int_0^{x^1}g^{2\g}(\tau^1,x^2,\ldots,x^n)\cdot \partial_\g a^{11}d\tau^1-
\int_0^{x^1} a^{1\g}\G^2_{\g1}d\tau^1+A(x^2,\ldots,x^n).
$}

\ \\[-4mm]

\noi
As $a^{12}(0,x^2,\ldots,x^n)\equiv A(x^2,\ldots, x^n)$, the function $A\in C^1$.

After differentiating the formula \eqref{eq16} by $x^2$ and using the law of commutation of derivatives and integrals, see \cite[p.~300]{kud}, we can see that
\begin{equation}\label{eq17}
\frac{\partial}{\partial x^2}\left\{\left(\hbox{$\int_0^{x^1}$}g^{2\g}(\tau^1,x^2,\ldots,x^n)d\tau^1\right)\cdot \smash\partial_\g(a^{\a\b}\cdot g_{\a\b})\right\}
\end{equation}
exists.
From \eqref{eq14} for $i=2$ and $k=c\neq2$ we obtain \
$
\partial_c a^{c2}=\frac 12\partial_2 a^{22}+a^{c\d}\G^2_{\d c}+a^{2\d}\G^c_{\d c}-a^{2\d}\G^2_{\d 2}
$.
Using this formula   we can rewrite the bracket \eqref{eq17} in the following form
 $$
 \left\{\left(\hbox{$\int_0^{x^1}$}g^{2\g}(\tau^1,x^2,\ldots,x^n)d\tau^1\right)\cdot g_{2\g}\ \cdot\ \partial_2 a^{22}+f\ \right\},
 $$
where $f$ is a rest of this parenthesis, which is  evidently differentiable by $x^2$.

Because of the parenthesis and also the coefficients by $\partial_2a^{22} $ are differentiable with respect to $x^2$, if
$
\left(\int_0^{x^1}g^{2\g}(\tau^1,x^2,\ldots,x^n)d\tau^1\right)\cdot g_{2\g}\neq 0
$, then $\partial_{22}a^{22}$
must exist.

Using \eqref{eq3} this inequality is true for all $x$ in a neighborhood of the point~$p$ excluding the point for which $x^1=0$.

For these reasons in this domain exists the derivative $\partial_{22}a^{22}$ and also exist all second derivatives $a^{ij}$. This follows from the derivative of the formula \eqref{eq14}.

So $a^{ij}\in C^2$ and $\lambda^i\in C^1$, from the formula (\ref{eq7}b) it follows $\psi_i\in C^1$ and it means that $\Psi \in C^2$. From (\ref{eq7}a)  follows $\bar g^{ij}\in C^2$ and also $\bar g_{ij}\in C^2$.
This is a proof of the Theorem \ref{thm:3}.
\end{proof}

\end{document}